\documentclass[11pt]{article}

\usepackage{latexsym,amsmath,amsthm,amssymb,amscd}

\newfont{\bb}{msbm10 at 12pt}

\def\r{\hbox{\bb R}}

\def\s{\hbox{\bb S}}

\def\amb{\mathcal{N}}

\newcommand{\derv}[2]{\dfrac{d #1}{d #2}}

\newcommand{\set}[1]{\left\{#1\right\}}
\newcommand{\metag}[2]{g\left( #1 , #2 \right) }
\newcommand{\meta}[2]{\langle #1 , #2 \rangle }

\newcommand{\hmcf}{H_{\phi}}
\newcommand{\camb}{\overline{\nabla}}

\numberwithin{equation}{section}

\begin{document}


\theoremstyle{plain}\newtheorem{lemma}{Lemma}[section]
\theoremstyle{plain}\newtheorem{definition}{Definition}[section]
\theoremstyle{plain}\newtheorem{theorem}{Theorem}[section]
\theoremstyle{plain}\newtheorem{proposition}{Proposition}[section]
\theoremstyle{plain}\newtheorem{remark}{Remark}[section]
\theoremstyle{plain}\newtheorem{corollary}{Corollary}[section]

\begin{center}
\rule{15cm}{1.5pt} \vspace{.6cm}

{\Large \bf Halfspace type Theorems for Self-Shrinkers} \vspace{0.4cm}

\vspace{0.5cm}

{\large Marcos P. Cavalcante$\mbox{}^{\dag}$\footnote{The author is partially supported by CNPq-Brazil.}, Jos$\acute{\text{e}}$ M. Espinar$\mbox{}^{\ddag}$\footnote{The author is partially supported by Spanish MEC-FEDER Grant MTM2013-43970-P and CNPq-Brazil.}}\\
\vspace{0.3cm} 

\noindent $\mbox{}^{\dag}$ Universidade Federal de Alagoas\\ e-mail: marcos@pos.mat.ufal.br\vspace{0.2cm}

\noindent $\mbox{}^{\ddag}$Instituto Nacional de Matem\'{a}tica Pura e Aplicada\\ e-mail: jespinar@impa.br

\rule{15cm}{1.5pt}
\end{center}

\vspace{.5cm}

\begin{abstract}
In this short paper we extend the classical Hoffman-Meeks Halfspace Theorem \cite{DHofWMee90} to self-shrinkers, that is: 
\begin{quote}
{\it Let $P $ be a hyperplane passing through the origin. The only properly immersed self-shrinker $\Sigma$ contained in one of the closed half-space determined by $P$ is $\Sigma = P$.}
\end{quote}

Our proof is geometric and uses a catenoid type hypersurface discovered by Kleene-Moller \cite{SKleNMol14}. Also, using a similar geometric idea, we obtain that the only complete self-shrinker properly immersed in an closed cylinder  $ \overline{B ^{k+1} (R)} \times \r^{n-k}\subset \mathbb R^{n+1}$,  for some $k\in \{1, \ldots ,n\}$ and radius $R$, $R \leq \sqrt{2k}$, is the cylinder $\s ^k (\sqrt{2k}) \times \r^{n-k}$. We also extend the above results for $\lambda -$hypersurfaces.
\end{abstract}

\section{Introduction}

The classical Halfspace Theorem for minimal surfaces in $\r ^3$ asserts:

\begin{quote}
{\bf Theorem \cite{DHofWMee90}:} {\it A connected, proper, possibly branched, nonplanar minimal surface $\Sigma$ in $\r ^3$ is not contained in a halfspace.}
\end{quote}

The proof of the above result is a clever use of catenoids in $\r ^3$ and the Maximum Principle. We should point out that the Halfspace Theorem is not true for minimal hypersurfaces in $\r ^n$, $n \geq 4$, since the behavior at infinity of catenoids in $\r ^n$, $n\geq 4$, is quite different from catenoids in $\r ^3$.  

Since then, many other generalizations have been made, see \cite{LMaz13,HRosFSchJSpr13} and reference therein for recents works on the subject. The above cited results do not apply in our situation since they focus on surfaces. Here, we will work in any dimension.

\subsection{Self-shrinkers in $\r ^{n+1}$}

Let $X: (0,T) \times \Sigma \to \r ^{n+1}$ be a one parameter family of smooth hypersurfaces moving by its mean curvature, that is, $X$ satisfies
$$ \derv{X}{t} = - H \, N  $$where $N$ is the unit normal along $\Sigma _t = X(t, \Sigma)$ and $H$ is its mean curvature, here, $H$ is the trace of the second fundamental form. With this convention, if $\Sigma _t$ is oriented by the outer normal $N$, then $\Sigma _t$ is mean convex provided $H(\Sigma _t) \leq 0$. 

Self-similiar solutions to the mean curvature flow are a special class of solutions, they correspond to solutions that a later time slice is scaled (up or down depending if it is expander or shrinker) copy of an early slice. In terms of the mean curvature, $\Sigma$ is said to be a self-similar solution if, with the convention above, it satisfies the following equation
\begin{equation}\label{Eq:Selfsimilar}
H = c \, \meta{x}{N},
\end{equation}where $c=\pm \frac 1 2$, $x$ is the position vector in $\r ^{n+1}$ and $\meta{\cdot }{ \cdot }$ is the standard Euclidean metric. Here, if $c= -\frac 1 2$ then $\Sigma$ is said a {\bf self-shrinker} and if $c=+\frac 1 2$ then $\Sigma $ is called {\bf self-expander}.

First, we extend the Hoffman-Meeks Halfspace Theorem for self-shrinkers in any dimension. Our proof is geometrical and uses a catenoid type hypersurface discovered by Kleene-Moller \cite{SKleNMol14}. 

\begin{theorem}\label{Th:HalfSpace}
Let $P $ be a hyperplane passing through the origin. 
The only properly immersed self-shrinker contained in one of the closed half-space determined by $P$ is $\Sigma = P$.
\end{theorem}

Moreover, using a stability argument and the Maximum Principle, following the ideas of \cite{JEspRo09}, we are able to extend the above result to Halfspace type theorems for properly immersed self-shrinkers contained in a cylinder. Specifically:

\begin{theorem}\label{Th:Sphere}
The only complete self-shrinker properly immersed in an 
closed cylinder  $ \overline{B ^{k+1} (R)} \times \r^{n-k}\subset \mathbb R^{n+1}$,  for some $k\in \{1, \ldots ,n\}$ 
and radius $R$, $R \leq \sqrt{2k}$, is the cylinder $\s ^k (\sqrt {2k}) \times \r^{n-k}$.
\end{theorem}

\begin{remark}
An analytical proof of Theorem \ref{Th:HalfSpace} is done in \cite[Theorem 19]{SPigMRim14}.
The Theorem \ref{Th:Sphere}  is in the spirit of \cite[Theorem 2]{SPigMRim14}. 
Here, we replace the hypothesis on the boundedness of the mean curvature of
$\Sigma$, $|H|\leq \sqrt {2k}$, by properness on the cylinder. It would be interesting to know if both conditions are equivalent or not. 
\end{remark}

\subsection{Self-similar solutions as weighted minimal surfaces}

It is interesting to recall here that  self-similiar solutions to the mean curvature flow in $\r ^{n+1}$ can be seen as weighted minimal hypersurfaces in the Euclidean space endowed with the corresponding density (c.f. G. Huisken \cite{GHui90} or T. Colding and W. Minicozzi \cite{TColWMin12a,TColWMin12b}). We will explain this in more detail. 

In a Riemannian manifold $(\amb ,g)$ there is a natural associated measure, that is, the Riemannian volume measure $dv_g \equiv dv$. More generally, we can consider Riemannian measure spaces, that is, triples $(\amb , g , m)$, where $m$ is a smooth measure on $\amb$. Equivalently by the Radon-Nikod\'{y}m Theorem we can consider triples $(\amb , g , \phi)$, where $\phi \in C^\infty (\amb)$ is a smooth function so that $dm = e^\phi dv$. The triple $(\amb , g , \phi)$ is called a {\it manifold with density $\phi$}. 

One of the first examples of a manifold with density appeared in the realm of probablity and statistics, the Gaussian Space, i.e., the Euclidean Space endowed with its standard flat metric and the Gaussian density $e^{-\pi |x|^2}$ (see \cite{ACanVBayFMorCRos08,JEsp13} for a detailed exposition in the context of isoperimetric problems). In 1985, D. Bakry and M. \'{E}mery \cite{DBakMEme85}  studied manifolds with density in the context of difussion equations. They introduced the so-called {\it Bakry-\'{E}mery-Ricci tensor} in the study of diffusion processes given by
\begin{equation}\label{Eq:ricd}
 {\rm Ric}_\phi = {\rm Ric} - \camb ^2  \phi , 
\end{equation}where ${\rm Ric}$ is the Ricci tensor associated to $(\amb , g)$ and $\camb ^2$ is the Hessian with respect to the ambient metric $g$. However, manifolds with density appear in many other fields of mathematics.

M. Gromov \cite{MGro03} considered manifolds with density as {\it mm-spaces} and introduced the generalized mean curvature of a hypersurface $\Sigma \subset (\amb ,g , \phi)$ or {\it weighted mean curvature} as a natural generalization of the mean curvature, obtained by the first variation of the weighted area
\begin{equation}\label{Eq:Hd}
H_\phi = H - \metag{N}{\camb  \phi}. 
\end{equation}

\begin{definition}
Let  $\Sigma \subset (\amb , g , \phi)$ be an oriented hypersurface. We say that $\Sigma$ is $\phi -$minimal if and only if the weighted mean curvature vanishes, i.e., $H_\phi (\Sigma) =0$. More generally, an immersed hypersurface $\Sigma$ has constant weighted mean curvature $H_\phi (\Sigma) = H_0$ (see \cite{TColWMin12a,TColWMin12b,GHui90}).
\end{definition}

It is straightforward to check that self-shrinker (resp. self-expander) are weighted minimal hypersurfaces in  $(\r ^{n+1}, \meta{}{} , \phi)$ with density $\phi := -\frac{|x|^2}{4}$ (resp. $\widetilde \phi  := \frac{|x|^2}{4}$). Henceforth, we will denote $\r ^{n+1}_\phi = (\r ^{n+1}, \meta{}{} , \phi)$, where $\phi = -\frac{|x|^2}{4}$.

 \begin{remark}
Recently, Cheng and Wei \cite{CW} introduced the notion of $\lambda$-hypersurfaces in $\r^{n+1}$. 
We say that an oriented hypersurface   $\Sigma \subset \r^{n+1}$ is a \emph{$\lambda$-hypersurface} if it satisfies the equation
\begin{equation*}\label{lambda}
H+\frac 1 2\langle N, x\rangle = \lambda.
\end{equation*}

Note that a $\lambda -$hypersurface is nothing but a constant weighted mean curvature $H_\phi = \lambda$ hypersurface in $\r ^{n+1} _\phi$. 
 \end{remark}

For $\lambda-$hypersurfaces we prove:

\begin{theorem}\label{Th:HalfSpaceLambda}
Set $\lambda \in \r $. Let $P_ \lambda $ be a hyperplane  defined by $\{x_{n+1}=\lambda \}$ . 
The only properly immersed $\lambda$-hypersurface contained in $\{x_{n+1} \geq \lambda \}$  is $\Sigma = P_ \lambda $.
\end{theorem}

\begin{remark}
Note that Theorem \ref{Th:HalfSpaceLambda} is not true for constant mean curvature surfaces in the Euclidean Space. 
\end{remark}

Moreover, we show

\begin{theorem}\label{Th:SphereLambda}
The only complete  $\lambda$-hypersurface  properly immersed in an 
closed cylinder  $ \overline{B ^{k+1} (R)} \times \r^{n-k}\subset \mathbb R^{n+1}$,  for some $k\in \{1, \ldots ,n\}$ 
and radius $R$ satisfying $\frac R 2-\frac k R \leq \lambda$, is the cylinder $\s ^k (R) \times \r^{n-k}$, with
$R=\lambda + \frac 1 2\sqrt{4\lambda^2+8k}.$
\end{theorem}

\section{Some examples}

In this section we remind the properties of some important hypersurfaces in  $\r ^{n+1}_{\phi} $ that we will use later. Recall that $\r^{n+1}_\phi \equiv (\r ^{n+1}, \meta{}{} , \phi)$, where $\phi := -\frac{|x|^2}{4}$. Also, we denote
$$ H_{\phi} = H +\frac 1 2\meta{x}{N} .  $$

As we have seen above, weighted minimal hypersurfaces in $\r ^{n+1}_{\phi}$, $H_{\phi} \equiv 0 $, correspond to self-shrinkers in $(\r ^{n+1} , \meta{}{})$.

\subsection{Spheres}

Let $\s ^n (R)$ be the rotationally symmetric $n$-dimensional sphere centered at the origin 
of radius $R$.  Let $N$ denote the outward orientation. 
Then, the usual mean curvature with respect to the outward orientation is $H = -\frac{n}{R}$. 
Moreover, since the position vector and the outward normal point at the same direction, we have 
$\meta{x}{N} =R$. Therefore, 
$$ \hmcf = - \frac{n}{R}+ \frac R 2 ,$$hence, 

\begin{itemize}
\item $\s ^n (R)$ has constant $H_{\phi} =\frac R 2- \frac{n}{R} >0$ for $R>\sqrt{2n}$.
\item $\s ^n (\sqrt{n})$ is a self-shrinker.
\item $\s ^n (R)$ has constant $H_{\phi} =\frac R 2- \frac{n}{R} <0$ for $R<\sqrt{2n}$.
\end{itemize}

\subsection{Hyperplanes}

Let $P_t$, $t \in \r$,  be the hyperplane given by 
$$ P_t := \set{  x_{n+1} =t } $$and we consider the upwards orientation $N_t = e_{n+1}$. 

As hypersurface in the Euclidean Space is a minimal hypersurface, that is, $H=0$. Moreover, the position vector along $P_t$ can be writen as $x := X + t e_{n+1}$, where $X$ is orthogonal to $e_{n+1}$. Thus, $\meta{x}{N_t} = t$ along $P_t $ for all $t \in \r$. So

\begin{itemize}
\item $P_t$ has constant $H_{\phi} =\frac t 2 >0$ for $t>0$.
\item $P_0$ is a self-shrinker.
\item $P_t$ has constant $H_{\phi} =\frac t 2 <0$ for $t<0$.
\end{itemize}

\begin{remark}
Note that at the highest point of the self-shrinker sphere $\s ^n (\sqrt{2n})$, we have that $P_{\sqrt{2n}}$ is 
above $\s ^n (\sqrt{2n})$, they are tangent at one point and both normal, as we have considered here, 
point at the same direction. But this not contradicts the Maximum Principle since 
$\hmcf (P_{\sqrt{2n}}) > \hmcf (\s ^n (\sqrt{2n})) =0$.
\end{remark}

\subsection{Cylinders}

Consider the cylinder centered at the origin given by $C_R^k := \s ^k (R) \times \r^{n-k}$, 
$1 \leq k \leq n$. As we did before, we know that $H(C_ R ^k) = \frac{k}{R}$ and 
$\meta{x}{ N_{k, R}} = R $, where $N_{k,R}$ is the outward orientation. Therefore, 
Therefore, 
$$ \hmcf (C^k_R)= \frac R 2- \frac{k}{R} ,$$hence, 

\begin{itemize}
\item $C^k_R$ has constant $H_{\phi} =\frac R 2- \frac{k}{R} >0$ for $R>\sqrt{2k}$.
\item $C^k_{\sqrt{k}}$ is a self-shrinker.
\item $C^k_R$ has constant $H_{\phi} =\frac R 2- \frac{k}{R} <0$ for $R<\sqrt{2k}$.
\end{itemize}

\subsection{Half Catenoid}

Here, we will describe a rotationally symmetric example that is of capital importance in our work. These are the rotationally symmetric self-shrinkers contained in a halfspace, embedded and with boundary on the hyperplane that defines the halfspace. This example is given by (see \cite[Theorem 3]{SKleNMol14})

\begin{equation}\label{Eq:Catenoid}
\begin{matrix}
\psi_\theta :&  [0 , +\infty ) \times \s ^{n-1} & \to &\r^{n+1} \equiv  \r ^n \times \r \\
               &    (t, \omega) & \to & (u_\theta (t) \omega , -t)
\end{matrix}
\end{equation}where $u_\theta : [0, +\infty ) \to \r^+$ has the following properties:
\begin{enumerate}
\item $u_\theta (t)> \theta \, t$ and $u_\theta (0) < \sqrt{2(n-1)}$.
\item $ \frac{u_\theta (t)}{t} \to \theta $ and $u' _\theta (t) \to \theta $ as $t \to + \infty$.
\item $u_\theta$ is strictly convex and $0< u'_\theta <\theta $ holds on $[0, +\infty)$.
\end{enumerate}

Moreover, its normal vector field is given by 
$$ N_\theta = \frac{1}{(1+(u'_\theta)^2)^{1/2}} (\omega , u' _\theta ) ,$$and, since $\mathcal C _\theta := \psi_\theta ([0,+\infty) \times \s ^{n-1})$ is a self-shrinker for all $\theta >0$, we have 
$$ \hmcf (\mathcal C _\theta ) = 0 .$$

One important observation is the following:

\begin{remark}\label{Remark}
The {\it half-catenoids} $\mathcal C _\theta $ interpolates between the plane 
$P_0 := \set{ x_{n+1} =0 }$ and the half-cylinder $C^{n-1}_{\sqrt{2(n-1)}} \cap \set{x_{n+1} \leq 0}$. Actually, 
\begin{itemize}
\item  $\mathcal C _\theta \to C^{n-1}_{\sqrt{2(n-1)}} \cap \set{x_{n+1} \leq 0}$ as $\theta \to 0$.
\item  $\mathcal C _\theta  \to P_0 $ as $\theta \to +\infty$.
\end{itemize}
\end{remark}

\section{Proof of Theorem \ref{Th:HalfSpace}}

We will argue by contradiction, so assume that $\Sigma \subset \r^{n+1}_{\phi}$ is a properly immersed self-shrinker contained in a halfspace determined by $P_0$ and $\Sigma$ is not $P_0$. We can assume that $\Sigma \subset \set{x_{n+1}\geq 0}$.

First, note that the function $h : \Sigma \to \r$, given by $h(p)= \meta{p}{e_{n+1}}$, can not have a minimum. Otherwise, there would exist a point $p_0$ so that $h_0= h(p_0) \leq h (p)$. This implies that $\Sigma$ and $P_{h_0}$ have a contact point at $p_0$, $\Sigma$ is above $P_{h_0}$ (with respect to the upward orientation) and $\hmcf (\Sigma) < \hmcf (P_{h_0})$. This contradicts the Maximum Principle. 

Therefore, we can assume that $\Sigma$ approaches some hyperplane $P_t$, $t \geq 0$, at infinity. 



Since $\Sigma $ is proper, there exists $\epsilon >0$ so that $D(\sqrt{2(n-1)})\times [0, t+\epsilon] \cap \Sigma = \emptyset$, where $D(\sqrt{2(n-1)}) \subset P_0$ is the Euclidean $(n-1)-$ball centered at the origin of radius $ \sqrt{2(n-1)}$. 

Now, we translate upwards the family of half-catenoids $\mathcal C _\theta$. We denote
$$ \mathcal C _{\theta ,s} := \mathcal C _\theta + s \, e_{n+1} , $$for $s \geq t$.

One can easily see that the normal $N_{\theta , s}$ along $\mathcal C _{\theta , s}$ 
satisfies $\meta{N_{\theta , s}}{e_{n+1}} >0$ and 
$$ \hmcf (\mathcal C _{\theta ,s}) = \frac{s \, u'_\theta}{(1+(u'_\theta)^2)^{1/2}} >0,$$which is positive along $\mathcal C _{\theta , s}$. 

Therefore, take $s \in (t, t+\epsilon)$, then $\partial \mathcal C _{\theta ,s}$ does not touch 
$\Sigma$ for all $\theta \in (0 , +\infty)$. 
Note that $\mathcal C _{\theta ,s} \to C^{n-1}_{\sqrt{2(n-1)}}\cap \set{ x_{n+1} \leq s}$ 
as $\theta \to 0$ and $\mathcal C _{\theta ,s} \to P_s$ as $\theta \to + \infty$. 
Also, note that $\mathcal C _{\theta, s}$ is not asymptotic to any hyperplane $P_t$, $t \leq s$. In fact, $\mathcal C_{\theta ,s}$ is asymptotic to a cone for $\theta >0$.

Therefore, since $\Sigma $ approaches $P_t$, there exists $\theta _0 $ so that $\mathcal C _{\theta _0 , s} $ has a finite first contact point with $\Sigma$ as $\theta $ increases from $0$. Clearly, both normals point upwards and $\Sigma$ is above $\mathcal C _{\theta _0 , s}$, but $\hmcf (\mathcal C _{\theta _0 , s}) > \hmcf (\Sigma) =0$, which contradicts the Maximum Principle. 

Thus, $\Sigma \equiv P_0$. This finishes the proof.

\section{Proof of Theorem \ref{Th:Sphere}}

We split the proof in two cases: the case of the ball, which is simple, and the case of a non degenerated cylinder, where we use a stability argument following ideas in \cite{JEspRo09}.

\subsection{Self-shirinkers in a ball}

Since $\Sigma$ is proper in $\overline{\mathbb{B}^{n+1} (R)}$, where $\mathbb{B}^{n+1} (R)$ is the Euclidean $(n+1)-$ball centered at the origin of radius $R$, we have that $\Sigma$ is compact.  In particular, there exists $p \in \Sigma$ such that 
$$
 d(p, \s ^n (R)) = {\rm dist}(\Sigma , \s^n(R)) .
$$

Therefore, we may choose $R'\leq R$ such that $\Sigma $ and $\s^n (R')$ are tangent at $p$. 
Since the weighted mean curvature of $\s ^n(R')$ is given by $H_{\phi}=\frac{R'} 2- \frac{n}{R'} \leq 0$, the Maximum Principle implies that $R'= R =\sqrt{2n} $ and $\Sigma = \s ^n(\sqrt {2n})$.

\subsection{Self-shirinkers in cylinders}

We will argue by contradiction. So, assume $\Sigma$ is not $\s ^k (\sqrt{2k})\times \r ^{n-k}$.

We start with an important lemma about the stability of  cylinders as self-shrinkers. 
We recall (see for instance \cite{JEsp13}) that 
the first variation of the weighted area  funcional of an immersed hypersurface $\Sigma$
in $\r^{n+1}_\phi$ is given by  the weighted mean curvature $H_\phi$, 
while  the second variation is given by the following Jacobi operator:
$$
J_{\phi}u=\Delta u-\frac 1 2 \langle\, x \,, \nabla u \rangle+ (|A|^2+{\rm Ric}_\phi(N))u, \quad u\in C^\infty_0(\Sigma).
$$

It is easy to see that $J_\phi$ is selfadjoint with respect to the weighted $L^2 -$norm given by 
$$ (u,v)_{L^2_\phi} = \int _\Sigma uv e^\phi \, dv , \, \, u,v \in C^\infty _0 (\Sigma) .$$

We say that $\Sigma$ is \emph{stable} in $\r^{n+1}_\phi$ if the Jacobi operator $J_\phi$ is nonpositve on $\Sigma$, that is, if the quadratic form $Q_\phi( u,v) = (J_\phi u ,v)_{L^2 _\phi}$ is nonpositive for all $u,v \in C^\infty _0 (\Sigma)$. Otherwise we say that $\Sigma$ is \emph{unstable}. Notice that if there exist a positive constant $c_0$ and a non trivial function 
$u \in C^\infty_0(\Sigma) $ such that $J_\phi u \geq c_0 u$, then $\Sigma$ is unstable. In the latter case, small variations of $\Sigma$ given by $u$ decrease the weighted mean curvature. This is actually what happens with the self-shirinkers cylinders as we see bellow.
 
\begin{lemma}\label{Lem:Stable} 
For any $k\in \{1,\ldots, n-1\}$ and $R>0$, the cylinders $C^k_R\subset \r^{n+1}_\phi$ are unstables hypersurfaces with respect to the Jacobi operator $J_\phi$.
\end{lemma}
\begin{proof}
For cylinders $C^k_R$ we have
$$
|A|^2 = \frac {k}{R^2} \quad \text {and}\quad {\rm Ric}_\phi(N)=\frac 1 2.
$$

Given $r>0$ we  consider 
$$
u(p,\overline t)= u(\overline t)=
 \prod_{i=1}^{n-k}\cos\Big(\frac \pi r t_i\Big) 
$$ 
as a test function, where $p\in \s^{k}$ and $ \overline t = (t_1,\ldots,t_{n-k}) \in \big(-\frac r 2,\frac r 2 \big)^{n-k}$.
Then a direct computation yields 
$$
-\frac 1 2\meta {x}{\nabla  u} = \frac {\pi}{2r}\sum_{i=1}^{n-k} t_i\sin \Big(\frac{\pi}{r}t_i \Big)\prod _{j=1, \, j\neq i }^{n-k} \cos\Big(\frac \pi r t_j\Big)\geq 0,
\quad {\rm for } \, (t_1,\ldots, t_{n-k})\in \Big(-\frac r 2,\frac r 2 \Big)^{n-k},
$$
and $\Delta u  = -\frac{n-k}{r^2}\,\pi^2\, u$.
Thus we get
$$
J_\phi u \geq \Big(\frac 1 2 + \frac {k}{R^2} -\frac{n-k}{r^2}\pi^2  \Big)u.
$$
Finally, we can choose $r>0$ big enough so that 
$$
c=  c(k,R,n,r) := \frac 1 2 + \frac {k}{R^2} -\frac{n-k}{r^2}\pi^2 >0.
$$
This concludes the proof.
\end{proof}

Now fix  $k\in \{1, \ldots ,n-1\}$  and we assume that $\Sigma$ is a hypersurface properly immersed in the closed cylinder $ \overline{B^{k+1} (R)} \times \r^{n-k}\subset \mathbb R^{n+1}$. If $dist (\Sigma, C^k_R)$ is attained at a finite point $p\in\Sigma$, then we can apply the Maximum Principle using as barriers the cylinders $C^k_{R'}$, $R'\leq R$, to get a contradiction. So the distance is not attained at a finite point and without loss of generality we may assume that $dist (\Sigma, C^k_R)=0$.

Let $r>0$ big enough such that the function 
$$ u(p, \bar t) =\prod_{i=1}^{n-k}\cos\Big(\frac \pi r t_i\Big) , $$given in Lemma \ref{Lem:Stable}, satisfies 
$$ J_\phi u \geq c u  $$for some positive constant $c$. 

Now, we consider the compactly supported variation normal of $C^k_R$ given by the vector field $X= u \tilde N$, here $\tilde N$ is the normal along $C^k _R$ and it is given by 
$$ \tilde N (p, \bar t)  = \frac{1}{R}(p, 0) \, , \,\, (p,\bar t) \in \s ^k (R)\times \r ^{n-k} = C^k_R .$$

The family of compact with boundary hypersurfaces associated to such variation is given by the 
normal variation of a peace of $C^k_R$ in the direction of $u$. Namely, 
$$
\Sigma_s=\Big\{(p,\overline t) + \frac{s}{R}u(\overline t)(p,0): 
(p,\overline t)\in \s^k(R)\times \Big(-\frac r 2,\frac r 2 \Big)^{n-k} \Big\},   \quad s\in (-\epsilon, \epsilon) ,
$$for some $\epsilon >0$ small enough.

On the one hand, we should note that $\partial \Sigma _s \subset C^k_R$ and therefore $\partial \Sigma _s \cap \Sigma = \emptyset$ for all $s \in (-\epsilon , \epsilon)$, the unit normal along $\Sigma _s$ pointing outwards. Also, a straightforward computation shows that 
$$ N_s(p,\bar t) = \frac{R}{\sqrt{R^2 +s ^2 \| \tilde \nabla u \|^2}}(p , -\frac{s}{R}\tilde \nabla u (\bar t)) ,$$where $N_s$ is the outward normal along $\Sigma _s$ and $\tilde \nabla u$ denotes the gradient in $\r ^{n-k}$.

Moreover, we know (see \cite{ACanVBayFMorCRos08}) that 
$$H' _\phi (0) = -J_\phi u \leq -c u <0 ,$$
which means that the weighted mean curvature of $\Sigma _s$ is strictly negative, 
i.e., $H^s_\phi (q) < 0 $ for all $q \in \Sigma _s$, for all $s \in (-\epsilon , \epsilon)$. Possibly we must shrink $\epsilon$. 

Since $\Sigma$ is proper we have that $\Sigma_s\cap \Sigma = \emptyset$, $\forall s\in (-\epsilon, 0)$. 
%
%
%
%
On the other hand, since $dist (\Sigma, C^k_R)=0$ and it is not attained we can choose a
sequence $q_j = (p_j, \overline t_j)\in \s^k(R)\times \r^{n-k}$  such that $\lim dist (\Sigma,  q_j)=0$ and $\lim \|\overline q_j\| = \infty$. Consider 
$$ v_j = \frac{q_j}{\| q_j\|} \to (0 , v_\infty) \text{ as } j\to +\infty,$$where $v_\infty \in \s^{n-k-1} \subset \r ^{n-k}$. We can assume that $v_\infty = (0,\ldots , 0 ,1)$, since the problem is invariant under rotations of the Euclidean Space.

The idea here is to translate $\Sigma_s$ in the direction of $v_\infty$ and find a first contact point and so, 
to apply the Maximum Principle at this point to get a contradiction. Let us consider the translated hypersurfaces
$$
\Sigma_{s, h}= \Sigma _s + h v _\infty ,   \quad s\in (-\epsilon, 0) , \, h \geq 0.
$$
 
As we did in Theorem \ref{Th:HalfSpace}, note that the mean curvature and the outward unit normal vector field of $\Sigma_s$ and  $\Sigma_{s, h}$ coincide at the corresponding points. Hence, we can compute the weighted mean curvature $H_\phi ^{s,h}$ of $\Sigma _{s,h}$ as 
$$ H_\phi ^{s,h}(q + h v_\infty) = H_\phi ^s (q) + h \meta{v_\infty}{N_s(q)} , \, q = (p,\bar t) \in \Sigma _s .$$

A straightforward computation (see Lemma \ref{Lem:Stable}) yields that 
$$  \meta{v_\infty}{N_s(q)} = s \frac{\pi  \, u(\bar t)}{r\sqrt{R^2 +\| \tilde \nabla u\|^2}} \tan\left( \frac{\pi}{r}t_{n-k}\right). $$

Since $s<0$, if the first point of tangency $\tilde q+hv_\infty$ occurs for $h>0$, then $t_{n-k}\in (0,r/2)$. If it it occurs for $h<0$, then $t_{n-k}\in (-r/2,0).$ 
In any case we have that $ h \meta{v_\infty}{N_s(q)} <0.$
%
%
Therefore,  in the first tangency point we get 
$$ H_\phi ^{s,h}(\tilde q+hv_\infty) \leq H_\phi ^s (\tilde q) <0 ,$$
which gives the desired contradiction by the Maximum Principle. 
Hence, $\Sigma = \s ^k (\sqrt{2k}) \times \r ^{n-k}$.

\section{Proof of Theorem \ref{Th:HalfSpaceLambda} and Theorem \ref{Th:SphereLambda}}

We first show:

\begin{lemma}\label{Lem:StableLambda}
For any $\lambda \in \r $, the hyperplane  $ P_\lambda =\{ x_{n+1} =\lambda \} \subset \r^{n+1}_\phi$ is unstable with respect to the Jacobi operator $J_\phi$.
\end{lemma} 
\begin{proof}
We argue as in Lemma \ref{Lem:Stable}. For hyperplanes $P_\lambda$ we have $
|A|^2 = 0 $ and ${\rm Ric}_\phi(N)=\frac 1 2$. Hence, given $r>0$ we  consider 
$$
u(t_1, \ldots , t_n)=  \prod_{i=1}^{n}\cos\Big(\frac \pi r t_i\Big) , \, (t_1,\ldots,t_{n}) \in \big(-\frac r 2,\frac r 2 \big)^{n},
$$ 
as a test function. Then a direct computation yields 
$$
J_\phi u \geq \Big(\frac 1 2  -\frac{n}{r^2}\pi^2  \Big)u.
$$
Finally, we can choose $r>0$ big enough so that 
$$
c := \frac 1 2 -\frac{n}{r^2}\pi^2 >0.
$$

\end{proof}

Thus, arguing as in Theorem \ref{Th:Sphere}, we can prove Theorem \ref{Th:HalfSpaceLambda} and Theorem \ref{Th:SphereLambda}.

\begin{remark}
Note that we could prove Theorem \ref{Th:HalfSpace} using Lemma \ref{Lem:StableLambda} and the argument in Theorem \ref{Th:Sphere}. Nevertheless, we prefer the proof given here using the Catenoid type hypersurfaces of Kleene-Moller.
\end{remark}

\end{document}